\documentclass[12pt]{amsart}
\usepackage{geometry} 
\newtheorem{theorem}{Theorem}[section]
\newtheorem{prop}{Proporsition}[section]
\newtheorem{cor}{Corollary}[section]

\newtheorem{lem}{Lemma}[section]

\newcommand{\bz}{\mathbb{Z}}
\newcommand{\bq}{\mathbb{Q}}

\newcommand{\dd}{\mathbf{dim}}

\newcommand{\bF}{\mathbf{F}}

\newcommand{\bhom}{\mathbf{Hom}}

\newcommand{\mc}{\mathcal{C}}
\title{Tate Conjecture And Finiteness of Abelian Varieties over Finite Field}
\author{Anningzhe Gao}
\date{} 

\begin{document}

\maketitle
\begin{abstract}
In this paper we will prove that Tate conjecture of abelian varieties over finite field is equivalent to the finiteness of isomorphism classes of abelian varieties with a fixed dimension. We give a different approach with Zarhin's result \cite{Zarhin}.
\end{abstract}
\tableofcontents

\section{Introduction}
Through the paper, $k$ will denote a finite field of $char\ k=p$ and $\bar{k}$ means the algebraic closure of $k$. In \cite{tate1966endomorphisms}, Tate proved the following famous theorem 
\begin{theorem}[Tate]{\label{Tateconj}}
Let $k=\bF_{q}$ be a finite field where $q$ is a power of a prime $p$. Let $A,B$ be two abelian varieties over $k$. Let $G=Gal(\bar{k}/k)$ be the absolute Galois group of $k$. If $l$ is a prime and $l\neq p$, then we have the isomorphism
$$\bhom_{AV}(A,B)\otimes \bz_l\cong\bhom_{\bz_l[G]}(T_l(A),T_l(B))$$
Here $T_l$ is the Tate module of an abelian variety.
\end{theorem}
In the proof of the theorem, Tate used the following fact: Fix an integer $d$, a prime $l$ and an abelian variety $A$ over $k$, the isomorphism classes of abelian varieties which admits a polarization of degree $d^2$ and an isogeny $B\rightarrow A$ of degree $l^n$ for some $n$ is finite. If we denote $g=\dd A$, then since there is an isogeny between $B$ and $A$, so $\dd B=g$. It is well known that over $k$ (a finite field), the isomorphism classes of abelian variety of dimension $g$ is finite, so we have 
$$Finiteness\ of\ isomorphism\ classes\ of\ abelian\ varieties\ over\ k\ with\ a\ fixed\ dimension$$
$$\Longrightarrow Tate\ conjecture\ of\ abelian\ varieties\ over\ k$$

In this paper we will prove the converse direction, i.e.

\begin{theorem}[Main Theorem]{\label{main}}
Tate conjecture of abelian varieties over $k$ implies that thaere are only finitely many abelian varieties of dimension $g$ over $k$.
\end{theorem}

This result is first proved by Zarhin \cite{Zarhin}. We will give a different approach to this result. 

The organization of this paper is as follows: In Section 2 we will recall some basic properties of abelian varieties over $k$, in Section 3 we will use Tate isogeny theorem to prove there are only finitely many isogeny classes of abelian varieties, in Section 4 and Section 5 we will show there are finitely many abelian varieties isogenous to a fixed abelian variety.

$\mathbf{Notations}$: We will use $k$ to represent a finite field of characteristic $p$, $\bar{k}$ its algebraic closure, $G=Gal(\bar{k}/k)$ the absolute Galois group. $\sigma$ is the Frobenius element. For a projective variety $X$ over $k$, we use $\pi_X$ to denote the Frobenius morphism of $X$.

I would like to thank my advisor, Martin Olsson, introduced me this interesting topic and lots of useful discussions. And Daniel Bragg and Thomas Preu for helpful discussion on some details.

\section{Some Basic Facts About Abelian Varieties}

In this section we recall the Tate module of an abelian variety and the $p-$divisible group.

Let $A$ be an abelian variety over $k$ with $\dd A=g$. Choose $l$ a prime number with $l\neq p$. We know that if $(p,n)=1$, the morphism $n:A\to A$ is a separable isogeny of degree $n^{2g}$, denote $A[n]=Ker(n:A\to A)$, then $A[n](\bar{k})\cong(\bz/n\bz)^{2g}$. By definition the Tate module $$T_l(A)=\varprojlim_n A[l^n](\bar{k})$$
We know $T_l(A)\cong \bz_l^{2g}$ non-canonically. The Galois group $G$ acts on $T_l(A)$ in a natural way. This action is continuous, since the Frobenius is an topological generator of $G$ so the action of  $\sigma$ determines the action of $G$. 
The Frobenius morphism $\pi_A:A\to A$ is a morphism in $\bhom_{AV}(A,A)$, by definition the image of $\pi_A$ under the isomorphism  
$$\bhom_{AV}(A,A)\otimes \bz_l\cong\bhom_{\bz_l[G]}(T_l(A),T_l(A))$$
is  $\sigma$.

For $\pi_A$, we define a function $P_{\pi_A}(n)=deg(n-\pi_A)$, then we know $P_{\pi_A}$ is a polynomial of degree $2g$ with $\bz$ coefficients. It is the same as the characteristic polynomial of $\sigma$ on $V_l(A)=T_l(A)\otimes\bq_l$. In particular the characteristic polynomial of $\sigma$ on $V_l(A)$ is independent of $l$.

In the case if $l=p$, since now $p:A\to A$ is not separate, things are a little different. We use the $p-$divisible group in this case. We define

$$A[p^\infty]=\varinjlim_n A[p^n]$$

To introduce the Tate $p-$conjecture, we need to use definition of Dieudonne ring and Dieudonne modules, for details, see \cite{waterhouse1968abelian} and \cite{pink2004finite}.

Let $D_k$ be the Dieudonne ring of $k$, it is a non-commutative associative $W(k)-$algebra ($W(k)$ is the ring of Witt vectors) with two generators $F,V$ satisfying the following conditions:
$$FV=VF=p$$
$$F(c)=\phi(c)F$$
$$cV=V\phi(c)$$
for any $c\in W(k)$. Here $\phi$ is the automorphism of $W(k)$ induced by the automorphism $x\to x^p$ on $k$. So if $k=\bF_p$, then $D_k$ is commutative.

By the standard procedure (see \cite{pink2004finite}) we can associate $A[p^\infty]$ with a $D_k$ module $M(A)$. It is a free $W(k)$ module of rank $2g$. Its $D_k$ action is uniquely determined by the action of $F$ (or $V$), then Tate proved
\begin{theorem}[Tate]{\label{Tatep}}
For two abelian varieties $A,B$ over $k$, under the above notation, we have a natural isomorphism
$$\bhom_{AV}(A,B)\otimes \bz_p\cong\bhom_{D_k}(M(B),M(A))$$
\end{theorem}

This theorem can be found in \cite{waterhouse1968abelian}. In particular, we have 
$$\bhom_{AV}(A,A)\cong\bhom_{D_k}(M(A),M(A))$$
In this case, if we denote $\sigma_A$ is image of $\pi_A$ (The Frobenius of $A$) under this isomorphism, then $\sigma_A=F^m$ if $k=\bF_{p^m}$. And the character polynomial of $\sigma_A$ is just $P_{\pi_A}$.

\section{The Finiteness of Isogenous Classes}
 In this section we will prove that there are finitely many isogeny classes of abelian varieties over $k$ of dimension $g$. We first recall the isogeny thoerem.
 
 \begin{theorem}[Tate]{\label{tateiso}}
 Given two abelian varieties $A,B$ over $k$, then 
 $$A\ and\ B\ are\ isogenous$$ $$\iff P_{\pi_A}=P_{\pi_B}$$
 $$\iff T_l(A)\otimes\bq_l\cong T_l(B)\otimes\bq_l\ as\ \bq_l[G]\ modules$$
 \end{theorem}
 So to consider the isogenous classes we just need to consider the characteristic polynomials of the Frobenius. But we have the following big theorem.
 \begin{theorem}[Weil Conjecture]{\label{Weil}}
 Let $A$ be an abelian variety over $k$ with dimension $g$, then $P_{\pi_A}$ is a monic polynomial with coefficients in $\bz$ with degree $2g$, and if $\alpha$ is a root of $P_{\pi_A}$, then for any Galois embedding $\eta:\bar{\bq}\to\bar{\bq}$ over $\bq$, we have $|\eta(\alpha)|=\sqrt{q}$, here $q$ is the number of elements in $k$.
 \end{theorem}
 With these two theorems, we can state
 \begin{cor}{\label{iso}}
 There are only finitely many isogenous classes of abelian varieties over $k$ of dimension $g$.
 \end{cor}
 \begin{proof}
 By theorem \ref{tateiso}, it suffices to prove there are only finitely many characteristic polynomials. Suppose $k=\bF_{q}$. If $$P(x)=\Sigma_{i=0}^{2g}a_ix^{2g-i}$$ is the characteristic polynomial of some abelian variety, and $\alpha_1,\alpha_2,...,\alpha_{2g}$ are roots of $P(x)$, then by theorem \ref{Weil}, $|\alpha_i|\leq\sqrt{q}$, so we have $$|a_s|=|\Sigma_{1\leq i_1<i_2<...<i_s\leq 2g}\alpha_{i_1}...\alpha_{i_s}|$$
 $$\leq\Sigma_{1\leq i_1<i_2<...<i_s\leq 2g}|\alpha_{i_1}...\alpha_{i_s}|$$
 $$\leq\Sigma_{1\leq i_1<i_2<...<i_s\leq 2g}\sqrt{q}^s$$
 $$\leq M\sqrt{q}^s$$
 for some $M$. So we know all $a_i$ are bounded by some number which only depends on the field $k$. But we know all $a_i$ are integers, so we only have finitely many choices, so there are only finitely many polynomials can be the characteristic polynomial of some abelian variety. So there are only finitely many isogenous classes.
 \end{proof}
 So to prove there are finitely many isomorphism classes it suffices to show every isogenous class of abelian varieties only contains finitely many isomorphism classes.
 
 \section{Some Calculus of the Tate Module}
 In this section we fix an abelian variety $A$ over $k$ with dimension $g$. $\pi_A$ will denote the Frobenius morphism of $A$, $P_{\pi_A}$ is its characteristic polynomial. Let $\mc$ be the isogenous class containing $A$. We will also use $\pi_A$ to mean the element in $\bhom_{\bz_l[G]}(T_l(A),T_l(A))$ under the Tate's isomorphism, which can be regarded as a $2g\times 2g$ matrix with element in $\bz_l$. 
 
The main proposition of this section is:
\begin{prop}{\label{mainprop}}
With the above data, there exists a positive integer $N$ which only depends on $A$ (we will see from the proof $N$ only depends on $\mc$), such that for any $B\in\mc$ and $l>N$, $T_l(B)\cong T_l(A)$ as $\bz_l[G]$ modules, and for $l<N$, the set $\{T_l(B)|B\in\mc\}$ (consider as $\bz_l[G]$ modules) is a finite set (we include the case $l=chark$, in which case we consider Dieudonne modules as in section 2).
\end{prop}
Before the proof, we first notice that we must have $T_l(A)\otimes\bq_l\cong T_l(B)\otimes\bq_l$ as $\bq_l[G]$ module. As we discussed, the $G$ action action on the Tate module is uniquely determined by the action of the Frobenius. So we can see $T_l(A)\cong T_l(B)$ as $\bz_l[G]$ modules if and only if $\pi_A$ and $\pi_B$ are conjugate by some matrix in $GL_{2g}(\bz_l)$ (not $GL_{2g}(\bq_l)$, they already conjugate by some matrix in $GL_{2g}(\bq_l)$ by Tate's isogeny theorem).

We separated the proof into two parts, consists of the following two lemmas. They are all purely linear algebra things.
\begin{lem}{\label{1}}
There exists a posotive $N$ such that for any abelian variety $B$ which is isogenous to $A$, and $l>N$, we can find basis of $T_l(B)$ and $T_l(A)$ such that the matrices of $\pi_A$ and $\pi_B$ will be the same.
\end{lem}
\begin{proof}
We know $A$ is isogenous to $A_1^{b_1}\times A_2^{b_2}\times...\times A_s^{b_s}$ where all $A_i$ are simple, non-isogenous with each other. We know for a simple abelian variety $A_i$, the characteristic polynomial $P_{\pi_{A_i}}$ of the Frobenius is a power of an irreducible polynomial, so $P_{\pi_{A_i^{b_i}}}$ is also a power of an irreducible polynomial. The characteristic polynomials $P_{\pi_{A_i^{b_i}}}$ should coprime with each other, let $K_i=\Pi_{j\neq i}P_{\pi_{A_j^{b_j}}}$ then by Bezout theorem, there exists $g_i\in\bq[x]$ such that

$$\Sigma_{i=1}^s g_i(x)K_i(x)=1$$

Then choose $N_0$ be a positive number such that if $l>N_0$, then all $g_i(x)\in\bz_l[x]$ (i.e. $l$ doesn't divide any denominators in $g_i$). And denote $M_i=K_i(\pi)T_l(A)$, then since all $g_i\in\bz_l[x]$, so if $l>N_0$, $T_l(A)=\oplus M_i$. And this $N_0$ only depends on the chosen isogenous class $\mc$, and on each $M_i$, the characteristic polynomial of $\pi_A$ is $P_{\pi_{A_i^{b_i}}}$, which is a power of an irreducible polynomial. Then we will concentrate on one $M_1$, i.e. we just assume $M_1=T_l(A)$, and we can see the similar procedure can be applied to all $2\leq i\leq s$ and prove the lemma in the general case. 

We know $\pi_A$ is an invertible matrix with coefficients in $\bz_l$. Let $\{\alpha_1,...,\alpha_t\}$ be the roots of $P_{\pi_A}$, then we have $P_{\pi_A}=((x-\alpha_1)...(x-\alpha_t))^e$ for some $e$ and $Q(x)=\Pi_{i=1}^t(x-\alpha_i)\in\bz[x]$ is irreducible. Then we define
$$P_i(x)=\Pi_{j\neq i}(x-\alpha_j)$$
for $1\leq i\leq k$. Since they don't have common factors, so we may choose $h_i(x)\in\bar{\bq}[x]$ such that $$\Sigma h_iP_i=1$$ Choose $N_1$ such that if $l>N_1$, then we have $h_i\in\bar{\bz}_l[x]$. Then we can see for any $v\in T_l(A)$, $v=\Sigma h_i(\pi_A)P_i(\pi_A)v$. Also it is easy to check $$L_i=P_i(\pi_A)T_l(A)$$
lies in $\alpha_i$ eigenspace (consider this over $\bar{\bz_l}$). Let $\bar{L_i}$ be the $\bar{\bz_l}$ linear expansion of $L_i$ in $T_l(A)\otimes\bar{\bz_l}$. Since every eigenspace of different eigenvalues are linearly independent, so we have $$T_l(A)\otimes\bar{\bz_l}=\oplus \bar{L_i}$$ 

Define $D=\Pi_{i\neq j}(\alpha_i-\alpha_j)^2$, then $D\in\bz$. We pick $u_1,...,u_e$ to be an integral basis of $\bar{L_1}$ over $\bar{\bz_l}$, such that $u_1,...,u_e$ can be represented by $u_i=P_1(\pi_A)(w_i)$ for  $w_1,...,w_e$ in $T_l(A)$ (This is true by linear algebra and the definition of $\bar{L_1}$). Define $v_i=\Sigma P_j(\pi_A)(w_i)$. Then we have $v_i\in T_l(A)$. We prove if $l>max(N_1,|D|)$, then $$\{v_1,\pi_A(v_1),...,\pi_A^{t-1}(v_1),v_2,...,\pi_A^{t-1}(v_2),...,v_e,...,\pi_A^{t-1}(v_e)\}$$ is an integral basis of $T_l(A)$. Since $v=\Sigma h_i(\pi_A)P_i(\pi_A)v$, so it suffices to show $P_i(\pi_A)(v)$ can be represented over $\bar{\bz_l}$ by these elements. From the Galois theory $P_i(x)=\phi(P_1(x))$ for some $\phi\in Gal(\bar{\bq_l}/\bq_l)$. By definition of $w_i$, $$P_1(\pi_A)(v)=\Sigma \beta_jP_1(\pi_A)(w_j)$$ for some $\beta_j\in\bar{\bz_l}$, so we have $$P_i(\pi_A)(v)=\Sigma \phi(\beta_j)P_i(\pi_A)(w_j)$$
Then we just need to show $P_i(\pi_A)(w_j)$ can be represented over $\bar{\bz_l}$ by these elements. This is solved by considering the system linear equations:
$$v_1=\Sigma P_i(\pi_A)(w_1)$$
$$\pi_A(v_1)=\Sigma \alpha_iP_i(\pi_A)(w_1)$$
$$....$$
$$\pi_A^{t-1}(v_1)=\Sigma \alpha_i^{t-1}P_i(\pi_A)(w_1)$$
then the matrix of this system of linear equations has determinant $D$, so by definition of $l$ and the Crammer's rule, $P_i(\pi_A)(w_j)$ can be represented over $\bar{\bz_l}$ by these elements. So $$\{v_1,\pi_A(v_1),...,\pi_A^{t-1}(v_1),v_2,...,\pi_A^{t-1}(v_2),...,v_e,...,\pi_A^{t-1}(v_e)\}$$ is an integral basis of $T_l(A)$ (Here we only proved every element can be represented integrally by these elements. But we have $te$ elements here and $te=2g=\dd T_l(A)\otimes\bq_l$, so they must form a basis). And the matrix of $\pi_A$ under this basis is uniquely determined by $P_{\pi_A}$, just denote this matrix by $C$ (independent of $A$). So if we set $N=max(N_1,|D|)$, then for $l>N_1$, we can choose a basis as above such that the Frobenius acts on $T_l(A)$ is represented by the matrix $C$. But this is independent of $A$, so we can do the same thing for $B$, so the matrices of $\pi_A$ and $\pi_B$ are the same. For the general case, we can find $N_i$ for each $M_i$, they are all only depend on $P_{\pi_A}$, so just choose $N=max(N_0,N_1,...,N_s)$, then from the above procedure, we can choose basis such that $\pi_A$ and $\pi_B$ have the same matrix when $l>N$. We proved the lemma.

\end{proof}
\begin{lem}{\label{2}}
With $N$ defined as above, for $l<N$, the set $\{T_l(B)|B\in\mc\}$ (consider as $\bz_l[G]$ modules) is a finite set (we include the case $l=chark$, in which case we consider Dieudonne modules as in section 2).
\end{lem}
\begin{proof}
We use the same idea as in Lemma \ref{1}. We collect them here.\\
Let $P_{\pi_A}=\Pi_{i=1}^s P_{\pi_{A_i^{b_i}}}$ where $P_{\pi_{A_i^{b_i}}}$ is a power of an irreducible polynomial. Set $K_i=\Pi_{j\neq i}P_{\pi_{A_j^{b_j}}}$ for $1\leq i\leq s$, then these $K_i(x)$ don't have common factors. So we have $g_i(x)\in\bq[x]$ such that 
$$\Sigma_{i=1}^s g_i(x)K_i(x)=1$$
Fix some $l<N$. Define $M_i=K_i(\pi_A)T_l(A)$. If we set $s_1$ to be the smallest integer such that $l^{s_1}g_i(x)\in\bz_l[x]$, then we have $$l^{s_1}T_l(A)\subseteq\oplus M_i\subseteq T_l(A)$$ Be careful this $s_1$ only depends on $P_{\pi_A}$ and $l$. Then $\pi_A$ acts on $M_i$, and its characteristic polynomial is just $P_{\pi_{A_i^{b_i}}}$. Write $P_{\pi_{A_i^{b_i}}}=(\Pi_{j=1}^{r_i} (x-\alpha_{ij}))^{e_i}$. Define $$P_{ij}=\Pi_{n\neq j} (x-\alpha_{in})$$
for $1\leq i\leq s$ and $1\leq j\leq r_i$. Then by Bezout's theorem, we may find $h_{ij}\in\bar{\bq}[x]$ such that $$\Sigma_{j=1}^{r_i} h_{ij}P_{ij}=1$$
Define $$L_{ij}=P_{ij}(\pi_A)M_i$$
and $\bar{L_{ij}}$ be the $\bar{\bz_l}$ expansion of $L_{ij}$ in $T_l(A)\otimes\bar{\bz_l}$. Then $\bar{L_{ij}}$ is a free module of rank $e_i$. Choose $\{w_{i1},...,w_{ie_i}\}$ such that $P_{i1}(\pi_A)(w_{ij})$ $1\leq j\leq e_i$ is an integral basis of $L_{i1}$. Define $$v_{ij}=\Sigma_{n=1}^{r_i}P_{in}(\pi_A)(w_{ij}),\ 1\leq i\leq s,\ 1\leq j\leq e_1$$
Define $N_i$ to be the submodule of $M_1$ generated by
$$\{v_{i1},\pi_A(v_{i1}),...,\pi_A^{r_i-1}(v_{i1}),.....,v_{ie_i},...,\pi_A^{r_i-1}(v_{ie_i})\}$$ 
Let $D=(\Pi_{i=1}^{t}\Pi_{1\leq j,k\leq r_i}(\alpha_{ij}-\alpha_{ik}))^{2(e_1+e_2+...+e_s)}$, and choose $s_2$ to be the smallest number such that $$l_{s_2}h_{ij}\in\bz_l[x],\frac{l^s}{D}\in\bz$$
Then similar to the proof in Lemma \ref{1} we can see
$$l^{s_2}(\oplus M_i)\subseteq\oplus N_i\subseteq\oplus M_i$$
Then we have
$$\oplus N_i\subseteq T_l(A)\subseteq l^{-s_1-s_2}\oplus N_i$$
Note that the matrix of $\pi_A$ on $\oplus N_i$ is only determined by $P_{\pi_A}$ in the chosen basis. Also the action of $\pi_A$ on $T_l(A)$ is induced from $l^{-s_1-s_2}\oplus N_i$. But $l^{-s_1-s_2}\oplus N_i/\oplus N_i$ is a finite set, so we proved the finiteness of $\{T_l(B)|B\in\mc\}$ if $l\neq p$.

The $l=p$ case is similar as we can see we can do the similar calculus for $W(k)$ module $M(A)$ with the action $\pi_A=L^m$ if $k=\bF_{p^m}$. Then we can see that the set of $M(A)$ with the action of $\pi_A$ is finite, but for fixed $\pi_A$, there are only finitely many choices of $L$ since they must be semi-simple. So we have the set of $M(A)$ with $D_k$ action is a finite set.
\end{proof}

By Proposition \ref{mainprop}, to prove the finiteness of isomorphism classes of abelian varieties, it suffices to show for a fixed abelian variety $A$ of dimension $g$, the set
$$\{B\ an\ abelian\ variety,\ T_l(B)\cong T_l(A)\ as\ \bz_l[G]\ modules\ for\ all\ prime\ l\}$$
is a finite set. Here we include the case $l=p$, which we consider the $D_k$ module $M(A)$. We will show this in the next section.
\section{Finish the Proof}
In this section, we will show the set in the previous section
$$\{B\ an\ abelian\ variety,\ T_l(B)\cong T_l(A)\ as\ \bz_l[G]\ modules\ for\ all\ prime\ l\}$$
is finite with the fixed $A$.

\begin{lem}{\label{isogeny}}
For a prime $l\neq p$, if $T_l(A)\cong T_l(B)$ as a $\bz_l[G]$ module, then there exists an isogeny $\pi:B\to A$ with $(deg\pi,l)=1$.
\end{lem}
\begin{proof}
From Tate conjecture, we have the following isomorphism
$$\bhom_{AV}(B,A)\cong\bhom_{\bz_l[G]}(T_l(B),T_l(A))$$
If $\sigma:T_l(B)\to T_l(A)$ the isomorphism, then since $\bz$ is dense in $\bz_l$, so we may find an isogeny $\pi:B\to A$ such that the image of $\pi$ is close to $\sigma$. So the image of $\pi$ is also an isomorphism. Then set $N=Ker\pi$, so we have the exact sequence
$$0\to T_l(B)\to T_l(A)\to N_l\to 0$$
here $N_l$ means the sylow $l$ group of $N$, see \cite{van2007abelian} Prop. 10.6. Since $\pi$ induces isomorphism between Tate modules, so we must have $N_l=\{0\}$, so $(deg\pi,l)=1$.
\end{proof}
\begin{lem}{\label{isop}}
The same holds for $l=p$ case.
\end{lem}
\begin{proof}
The proof is really similar, the exact sequence is 
$$0\to N_p\to B[p^\infty]\to A[p^\infty]\to 0$$
see \cite{van2007abelian} 10.17
\end{proof}
We can conclude the previous lemma into one property:
\begin{prop}{\label{m}}
Fix an abelian variety $A$. If there exists an abelian variety $B$ such that $T_l(B)\cong T_l(A)$ for $l\neq p$ and $M(B)\cong M(A)$ as $D_k$ modules, then for any prime $l$ (maybe $l=p$), we have an isogeny $\pi_l:B\to A$ such that $deg(\pi)$ is coprime to $l$. 
\end{prop}
We need a technique lemma.

\begin{lem}{\label{tech}}
If we have two abelian varieties $A$ and $B$ and two isogenies $\pi_1:B\to A$ and $\pi_2:B\to A$. If we have two integers $m_1,m_2$ such that $(m_1,m_2)=1$ and $(m_1,deg\pi_1)=(m_2,deg\pi_2)=1$, then we have an isogeny $\pi:B\to A$ such that $(deg\pi,m_1m_2)=1$.
\end{lem}
\begin{proof}
Set $\pi=m_2\pi_1+m_1\pi_2$. First we show $\pi$ is an isogeny. So pick some $l\mid m_1$, then consider the image of $\pi$ in $\bhom_{\bz_l[G]}(T_l(B),T_l(A))$ under the Tate isomorphism. We can see by condition $m_2\pi_1$ under this isomorphism induces an isomorphism since $(m_2deg\pi_1,l)=1$, and $\pi$ and $m_2\pi_1$ are differ by $l$ times some homomorphism, so we have $\pi$ is an isomorphisms of Tate modules, so $\pi$ is an isogeny.

If $(deg\pi,m_1m_2)\neq1$, then there exists some $x\in Ker\pi\cap B[m_1m_2]$. Then by replacing $x$ by some multiple, we may assume there exists some prime factor $l$ of $m_1m_2$, just say a prime factor of $m_1$ (the case of $m_2$ is the same), such that $x\neq 0$, $lx=0$ and $x\in Ker\pi$. But then we have $$0=\pi(x)=m_2\pi_1(x)+m_1\pi_2(x)=m_2\pi_1(x)$$
so $x\in Ker(m_2\pi_1)$, so $x\in Ker(m_2\pi_1)\cap B[l]=\{0\}$, which is a contradiction. So this $\pi$ satisfies our requirements.
\end{proof}

Now we comes to the last lemma.

\begin{lem}{\label{mainlem}}
Fix an abelian variety $A$. If there exists an abelian variety $B$ such that $T_l(B)\cong T_l(A)$ for $l\neq p$ and $M(B)\cong M(A)$ as $D_k$ modules, then $B$ is a direct component of $A\times A$. Here direct component means we have an abelian variety $C$ such that $B\times C\cong A\times A$
\end{lem}
\begin{proof}
Choose any isogeny $\pi_1:B\to A$, then by Lemma \ref{tech}, Prop \ref{m} and the induction procedure, we may find an isogeny $\pi_2$ such that $(deg\pi_1,deg\pi_2)=1$, then we have an embedding 
$$g:B\to A\times A$$
$$b\to(\pi_1(b),\pi_2(b))$$
Then we may find isogenies $\phi_1:A\to B$ and $\phi_2:A\to B$ such that $$\phi_1\pi_1=deg\pi_1$$  $$\phi_2\pi_2=deg\pi_2$$
since $(deg\pi_1,deg\pi_2)=1$, so there exists $s,t\in\bz$ such that $sdeg\pi_1+tdeg\pi_2=1$, define $$f:A\times A\to B$$
$$(a_1,a_2)\to s\phi_1(a_1)+t\phi_2(a_2)$$
Then we have $fg=1_B$, so $B$ is a direct factor of $A\times A$.
\end{proof}
Now we can finish our proof.
\begin{theorem}
There are only finitely abelian varieties of dimension $g$ over $k$.
\end{theorem}
\begin{proof}
By \cite{milne1986abelian} Theorem 18.7, for a fixed abelian variety, there are only finitely many direct components, then the theorem follows from Corollary \ref{iso}, Prop \ref{mainprop} and Lemma \ref{mainlem}.
\end{proof}
\bibliographystyle{abbrv}
\bibliography{MyCitation}

\end{document}